\documentclass[a4paper]{amsart}
\usepackage{amsmath}
\usepackage{amssymb}
\usepackage[arrow, matrix, curve]{xy}
\usepackage{comment}
\usepackage{stmaryrd}
\usepackage{graphicx}

\def\to{\rightarrow}

\newtheorem{theorem}{Theorem}[section]

\newtheorem{prop}[theorem]{Proposition}
\newtheorem{cor}[theorem]{Corollary}

\theoremstyle{definition}
\newtheorem{definition}[theorem]{Definition}
\newtheorem{example}[theorem]{Example}

\theoremstyle{remark}
\newtheorem{remark}[theorem]{Remark}

\numberwithin{equation}{section}

\makeatletter
\newcommand{\xyR}[1]{%
\makeatletter
\xydef@\xymatrixrowsep@{#1}
\makeatother
} 

\makeatletter
\newcommand{\xyC}[1]{%
\makeatletter
\xydef@\xymatrixcolsep@{#1}
\makeatother
}

\begin{document}

  \title{Shuffles of trees}


\author[E. Hoffbeck]{Eric Hoffbeck}
\address{Universit\'e Paris 13, Sorbonne Paris Cit\'e,
LAGA,
CNRS, UMR 7539,
99 avenue Jean-Baptiste Cl\'ement,
F-93430, Villetaneuse, France}
\email{hoffbeck@math.univ-paris13.fr}

\author[I. Moerdijk]{Ieke Moerdijk}
\address{Department of Mathematics, Utrecht University, PO BOX 80.010, 3508 TA Utrecht, The Netherlands}
\email{i.moerdijk@uu.nl}




%
\begin{abstract}
We discuss a notion of shuffle for trees which extends the usual notion of a shuffle for two natural numbers. We give several equivalent descriptions, and prove some algebraic and combinatorial properties. In addition, we characterize shuffles in terms of open sets in a topological space associated to a pair of trees. Our notion of shuffle is motivated by the theory of operads and occurs in the theory of dendroidal sets, but our presentation is independent and entirely self-contained. 
\end{abstract}

\maketitle

%

\section*{Introduction}

For two natural numbers $p$ and $q$ the set of $(p,q)$-shuffles plays a central r\^ole in many parts of algebra, topology, probability theory   and combinatorics. For example, they occur in the description of the coalgebra and Hopf algebra structures on exterior or tensor algebras \cite{Sweedler}, and in the description of the Eilenberg-Zilber map for the homology of a product of two topological spaces \cite{MacLane}. The name shuffle refers back to the fact that the $(p,q)$-shuffles are shuffles of linear orders of length $p$ and $q$ (rather than just sets of cardinality $p$ and $q$),
like shuffling two decks of $p$ and $q$ cards respectively, as studied in~\cite{AD}.

The goal of this paper is to study a notion of shuffle of two trees, rather than just of linear orders. Several such notions already occcur in the literature, for example in the context of automata theory and formal languages \cite{Ito}. Our notion is different from these.
It specialises to the standard one if the two trees happen to be linear orders, and seems very natural from the point of view of non-deterministic programming semantics where the trees describe programs. Such shuffles of trees also enter in the description of a free resolution of the Boardman-Vogt tensor product of operads \cite{BV}, and related to this, play a crucial r\^ole in the homotopy theory of dendroidal sets \cite{CM}.

In this paper, we will present a purely combinatorial and self-contained discussion of this notion of shuffle of two trees, which is motivated by, but can be read completely independently from the theory of operads and dendroidal sets. In particular, we shall consider questions concerning shuffles which have not been addressed in that context, such as: What is the structure of the set of shuffles of two trees? How is the number of shuffles related to the size of the trees? etc. That second question has a very simple answer in terms of binomial coefficients in the linear case, but seems quite intractible for general trees, as we will explain.

The plan of the paper, then, is as follows. In a first section, we will give what we believe is the most accessible definition of a shuffle of two trees, and illustrate it by various examples. In a subsequent section, we prove that this definition is equivalent to several others, the most concise one being that a shuffle of two trees $S$ and $T$ is a maximal subtree of the product partial order $S \times T$ containing all the pairs of leaves (cf Proposition~\ref{Maximality}). In a third section, we discuss some aspects of the number of shuffles of two trees. It is an open question whether one can find a comprehensible closed formula expressing the number of shuffles of two trees $S$ and $T$ in terms of the size (height, width, etc) of $S$ and $T$. We present some upper and lower bounds, and show that in the case where $T$ is a linear tree of length $n$ and $S$ is fixed, this number is a polynomial in $n$ with rational coefficients, of which the degree and leading coefficient can be described quite simply in terms of the size of $S$. In a fourth  section, we will show that the set of shuffles of two trees $S$ and $T$ carries the natural structure of a distributive lattice, which is rather rigid in the sense that in most cases it has no automorphisms other than the ones coming from automorphisms of $S$ or $T$. This description also leads to the observation that shuffles can be composed.
Indeed, two shuffles, one  between $S$ and $T$, and the other between $R$ and $S$, naturally give rise to a third shuffle between $R$ and $T$. More technically, we prove that trees and shuffles between them form a category enriched in distributive lattices.
In a final section, mainly added for motivation and background, we very briefly discuss in which way shuffles of trees naturally occur in topology, in operad theory and in the theory of dendroidal sets.

\textbf{Acknowledgements:} We would like to thank James Cranch who wrote a computer program enumerating shuffles of some small trees, which led to the examples \ref{ExCranch1} and \ref{ExCranch2} in the paper. The first author is also indebted to Denis-Charles Cisinski and Gijs Heuts for discussions which improved his understanding of shuffles.
We also like to thank our home universities and NWO for supporting our mutual visits.

%

\section{Definition and first examples}\label{S:Def}

In this section we present several equivalent definitions of a shuffle of two trees.
What we mean by a ``tree'' in this context is a finite graph without cycles, whose external edges are open, i.e. connected to one vertex only.
One of these external edges is selected as the root, and the other external edges are called the leaves of the tree.
The chosen root provides an orientation, pictured downwards towards the root. Each vertex will have one outgoing edge (towards the root), and a strictly positive number of incoming edges to which we will refer as the valence of the vertex (see also Remark~\ref{Rem:Stumps}).
Here is a picture of the type of tree that we shall consider.

\begin{equation*}
\xymatrix@R=11pt@C=15pt{
&&&&&&&&\\
&&&&*{}&*{}&*{\bullet}\ar@{}[r]_{u \quad \quad }  \ar@{-}[u]_{i}& \\
&& *{\bullet}\ar@{}[l]^{\quad \quad y} \ar@{-}[ul]^{c}  \ar@{-}[ur]_d &&& *{\bullet}\ar@{}[r]_{z \quad \quad }  \ar@{-}[ul]^{f} \ar@{-}[u]_{g} \ar@{-}[ur]_h&  \\
&&&*{\bullet}\ar@{}[l]^{\quad \ x} \ar@{-}[ul]^b \ar@{-}[urr]_e&&  \\
&&&*{}\ar@{-}[u]_a 
}
\end{equation*}

This is a tree with five leaves, and root edge $a$. There are four vertices, of valence two, two, three and one. The vertex which has the root as output edge will be refered as the root vertex. For an arbitrary tree $T$, we will write $E(T)$ and $V(T)$ for its sets of edges and vertices, respectively. Moreover, we will denote by $r_T$ its root edge (and later in Section~\ref{S:Lattice} its root vertex).

\begin{remark}
When drawing a tree, the picture automatically provides the tree with a planar structure. We do not presuppose our trees to have any planar structure however, and a picture like the one above but with the leaves $c$ and $d$ interchanged, for example, represents the same tree.
\end{remark}

\begin{remark}
There are various ways to think about such a tree which are relevant for what follows. When reading the tree from top to bottom, we think of the edges as \textit{objects} and of the vertices as \textit{operations}, taking a finite (but strictly positive) number of input objects to an output object. In the picture above, the vertex $y$ is an operation taking the objects $c$ and $d$ as inputs and producing $b$ as output. Operations can be composed, for instance $x \circ y$ is an operation with inputs $c, d, e$ and output $a$. This way of viewing trees is common, for example, in the theory of operads, cf. Section~\ref{S:Last} below.\\
In some cases, it may be more suggestive to read the tree from bottom to top, and view the vertices as \textit{decomposition} operations. In the example above, the operation $x$ decomposed the root object into two objects $b$ and $c$, etc. This point of view is relevant in modelling non-determistic programming: the program $x$ transforms a state $a$ into either $b$ or $c$, but we have no way of knowing which one and have to carry both $b$ and $c$ along in the model.
\end{remark}

\begin{remark}
 A tree in which all vertices have valence one will be called a linear tree. It is the same thing as a linear order on the (non-empty) set of its edges, or a (possibly empty) linear order on its vertices. In particular, we will also consider the \textit{unit tree} $\eta$ in which the leaf and the root coincide. We will denote by $L_n$ the linear tree with $n$ vertices. 
Here are pictures of the linear tree $L_3$ with edges $a,b,c,d$, and of the unit tree $\eta$.
\begin{equation*}
\vcenter{\xymatrix@R=8pt@C=10pt{
&\\
&*{ \bullet} \ar@{-}[u]_{d}&& \\
&*{ \bullet} \ar@{-}[u]_{c} \\
&*{ \bullet} \ar@{-}[u]_{b} \\
&*{} \ar@{-}[u]_{a} \\
}}
 \ \  \text{ and } \ \ 
\vcenter{\xymatrix@R=10pt@C=14pt{
&\\
&*{} \ar@{-}[u] \\
}}
\end{equation*}  

\end{remark}

An arbitrary tree $T$ can also be viewed as a non-empty partial order on the set $E(T)$ of its edges. 
We will call a partial order of this form \textit{treelike}.
Thus, a treelike partial order has the following characteristic properties: it is a finite partially ordered set with a smallest element (the root), and with the property that the induced order on each ``down-segment'' of the form $\downarrow e=\{ d \, | \, d \leq e \}$ is linear. It is of course sufficient to ask this for each maximal element $e$, i.e. for each leaf of the tree. 
When $e$ is a leaf, we call $\downarrow e$ the \textit{branch} associated to $e$. 
Sometimes, we will also consider the induced treelike partial order  on the (possibly empty) set $V(T)$ of the vertices of the tree $T$. Notice, however, that one cannot reconstruct $T$ from this partial order, because a vertex on top of the tree can have multiple leaves. However, given a possibly empty treelike partial order $V$, there is always a minimal tree $T$ for which $V=V(T)$. Such trees will be called \textit{reduced}. We will come back to these reduced trees in Section~\ref{S:Lattice}.

\medskip

\subsection{Shuffles of linear trees}\label{ss:ShufflesOfLinearTrees}
There is a well-known notion of shuffle of two linear orders, much used and studied in combinatorics and in algebraic topology:
a shuffle of two linear orders $A$ and $B$ is simply a linear order on the disjoint union $A+B$ which agrees with the given orders on $A$ and on $B$. If $A$ has $m\geq0$ elements and $B$ has $n\geq0$ elements, there are exactly $\binom{m+n}{n}$ such shuffles.

Another way, more common in algebraic topology, is to view a shuffle of two linear orders $A$ and $B$, now assumed non-empty, as a  maximal linearly ordered subset within the product partial order $A \times B$. If $A$ is the order $0<1< \ldots <m$ and $B$ is the order $0<1< \ldots <n$, then such a linear order has $(0,0)$ as smallest element and $(m,n)$ as largest, and can be pictured as a ̀``staircase'' path through the rectangular grid.

$$\xymatrix@M=0pt@C=6mm@R=6mm{
\ar@{.}[d]\ar@{.}[r]  &\ar@{.}[d]\ar@{.}[r]  &\ar@{.}[d]\ar@{->}[r]^{ \quad \ \quad (m,n)}  &\ar@{.}[d]&\\
\ar@{.}[d]\ar@{.}[r]  &\ar@{.}[d]\ar@{->}[r]  &\ar@{.}[d]\ar@{.}[r]\ar@{->}[u]  &\ar@{.}[d]&\\
\ar@{.}[d]\ar@{.}[r]  &\ar@{.}[d]\ar@{.}[r]\ar@{->}[u]  &\ar@{.}[d]\ar@{.}[r]  &\ar@{.}[d]&\\
\ar@{.}[d]\ar@{->}[r]  &\ar@{.}[d]\ar@{.}[r]\ar@{->}[u]  &\ar@{.}[d]\ar@{.}[r]  &\ar@{.}[d]&\\
\ar@{.}[r]_{(0,0) \quad \ \quad}\ar@{->}[u]  &\ar@{.}[r]  &\ar@{.}[r]  &&
}$$

There are again  $\binom{m+n}{n}$ such paths, or shuffles. (Notice, however, that there is a shift in counting for these two points of view, caused by the fact that we can consider a linear tree as an order of $n\geq0$ vertices or of $n+1>0$ edges.)
If we picture $A=(0<1< \ldots <m)$ as a linear tree with $m+1$ edges and $m$ vertices, and similarly for $B$, then this path through the grid is another linear tree whose edges are now named $(i,j)$ with $0\leq i \leq m$ and $0\leq j \leq n$, and where each vertex looks like one of 
\begin{equation*}
\vcenter{\xymatrix@R=10pt@C=14pt{
&\\
&*{ \bullet} \ar@{-}[u]_{(i+1,j)}&& \\
&*{} \ar@{-}[u]_{(i,j)} \\
}}
 \ \  \text{ or } \ \ 
\vcenter{\xymatrix@R=10pt@C=14pt{
&\\
&*{ \circ} \ar@{-}[u]_{(i,j+1)}&& \\
&*{} \ar@{-}[u]_{(i,j)} \\
}}\end{equation*}  
We call these shuffles of linear trees \textit{classical}.

\begin{example}
 For the two linear orders 
\begin{equation*}
A: \vcenter{\xymatrix@R=10pt@C=14pt{
&\\
&*{\ \, \bullet z} \ar@{-}[u]^{3}&&&&& \\
&*{\ \, \bullet y} \ar@{-}[u]^{2} \\
&*{\ \, \bullet x} \ar@{-}[u]^{1} \\
&*{} \ar@{-}[u]^{0} \\
}}
 \ \ \ 
B: \vcenter{\xymatrix@R=10pt@C=14pt{
&\\
&*{\ \, \circ v} \ar@{-}[u]^{c} \\
&*{\ \, \circ u} \ar@{-}[u]^{b} \\
&*{} \ar@{-}[u]^{a} \\
}}
\end{equation*}  

 \noindent
 here are three ways of viewing one and the same shuffle:
\begin{equation*}
 \vcenter{\xymatrix@R=10pt@C=14pt{
&\\
&*{\circ} \ar@{-}[u]^{(3,c)}&& \\
&*{\bullet} \ar@{-}[u]^{(3,b)} \\
&*{\bullet} \ar@{-}[u]^{(2,b)} \\
&*{\circ} \ar@{-}[u]^{(1,b)} \\
&*{\bullet} \ar@{-}[u]^{(1,a)} \\
&*{} \ar@{-}[u]^{(0,a)} \\
}}
\text{ or } \quad  \quad
\vcenter{\xymatrix@M=0pt@C=12mm{
\ar@{.}[d]\ar@{.}[r]  &\ar@{.}[d]\ar@{.}[r] & \ar@{.}[d]\ar@{.}[r]^{ \quad \quad \quad (3,c)}  &\\
\ar@{.}[d]\ar@{.}[r]  &\ar@{.}[d]\ar@{->}[r]^{(1,b) \quad \quad \quad} & \ar@{.}[d]\ar@{->}[r]_{(2,b) \quad \quad (3,b)}   &\ar@{.}[d]\ar@{->}[u]   &\\
\ar@{->}[r]_{(0,a) \quad \quad \quad}  &\ar@{.}[r]_{(1,a) \quad \quad \quad}\ar@{->}[u]  &  \ar@{.}[r]  &\\
}}
\end{equation*}  
\noindent
  or the linear order $x<u<y<z<v$.
  
\end{example}

\bigskip

We now wish to extend this notion of shuffle to trees.

\subsection{Shuffles of trees}\label{ss:ShufflesOfTrees}

\begin{definition}\label{Def1}
Let $S$ and $T$ be two trees.
A shuffle of $S$ and $T$ is a tree $A$ for which the following four conditions hold:
\begin{enumerate}
\item The edges of $A$ are labelled by pairs $(s,t)$ where $s$ and $t$ are edges of $S$ and $T$, respectively.
\item The root of $A$ is labelled by the pair $(r_S, r_T)$ of root edges of $S$ and $T$.
\item The set of labels of the leaves of $A$ is in bijective correspondence with the cartesian product of the leaves of $S$ and those of $T$.
\item If $(s,t)$ is the label of an edge of $A$ which is not a leaf, then either the incoming edges of the vertex above $(s,t)$ are labelled $(s_1,t), \ldots, (s_m,t)$ where $s_1, \ldots, s_m$ are the incoming edges of the vertex above $s$ in $S$; or these incoming edges are labelled $(s,t_1), \ldots, (s,t_n)$ where $t_1, \ldots, t_n$ are the incoming edges of the vertex above $t$ in $T$.
\end{enumerate}
Notation: We will write $Sh(S,T)$ for the set of shuffles of the trees $S$ and $T$.
\end{definition}

Here is a picture to illustrate the last condition of the definition:
\begin{equation*}
S:\vcenter{\xymatrix@R=10pt@C=14pt{
&&\\
&*{ \circ}\ar@{}[u]|{\cdots} \ar@{-}[ul]^{s_1}\ar@{-}[ur]_{s_m}&& \\
&*{} \ar@{-}[u]_{s} \\
}}
 \ \  \ \ 
T: \vcenter{\xymatrix@R=10pt@C=14pt{
&&\\
&*{ \bullet}\ar@{}[u]|{\cdots} \ar@{-}[ul]^{t_1}\ar@{-}[ur]_{t_n}&& \\
&*{} \ar@{-}[u]_{t} \\
}}\end{equation*}  

\begin{equation*}
A:\vcenter{\xymatrix@R=10pt@C=14pt{
&&\\
&*{ \circ}\ar@{}[u]|{\cdots} \ar@{-}[ul]^{(s_1,t)}\ar@{-}[ur]_{(s_m,t)}&& \\
&*{} \ar@{-}[u]_{(s,t)} \\
}}
 \ \text{ or }  \ \ 
A:\vcenter{\xymatrix@R=10pt@C=14pt{
&&\\
&*{ \bullet}\ar@{}[u]|{\cdots} \ar@{-}[ul]^{(s,t_1)}\ar@{-}[ur]_{(s,t_n)}&& \\
&*{} \ar@{-}[u]_{(s,t)} \\
}}\end{equation*}

\begin{example}\label{ExampleWith14}
Here is a complete list of the 14 shuffles of the following two trees

\begin{equation*}
S:\vcenter{\xymatrix@R=10pt@C=14pt{
&&\\
*{\circ} \ar@{-}[u]_2 \\
*{\circ} \ar@{-}[u]_1 \\
*{} \ar@{-}[u]_0 \\
}}
 \  \text{ and } \  \ \ 
T:\vcenter{\xymatrix@R=10pt@C=14pt{
&&&\\
*{\bullet}\ar@{-}[u]^c&&*{\bullet}\ar@{-}[u]_e\\
&*{ \bullet} \ar@{-}[ul]^b\ar@{-}[ur]_d&& \\
&*{} \ar@{-}[u]_a \\
}}\end{equation*}  
\noindent taken from \cite{MW2}, Example 9.4.

\begin{center}
$\xyR{5pt}\xyC{5pt}\xymatrix{*{}\ar@{-}[d] &  & *{}\ar@{-}[d] &  & *{}\ar@{-}[d] &  & *{}\ar@{-}[d]\\
*{\bullet}\ar@{-}[dr] &  & *{\bullet}\ar@{-}[dl] &  & *{\bullet}\ar@{-}[dr] &  & *{\bullet}\ar@{-}[dl]\\
 & *{\bullet}\ar@{-}[drr] &  &  &  & *{\bullet}\ar@{-}[dll]\\
 &  &  & *{\circ}\ar@{-}[d]_{}\\
 &  &  & *{\circ}\ar@{-}[d]\\
 &  &  & *{}\\
}
$~~~~~$\xyR{5pt}\xyC{5pt}\xymatrix{*{}\ar@{-}[d] &  & *{}\ar@{-}[d] &  & *{}\ar@{-}[d] &  & *{}\ar@{-}[d]\\
*{\bullet}\ar@{-}[dr] &  & *{\bullet}\ar@{-}[dl] &  & *{\bullet}\ar@{-}[dr] &  & *{\bullet}\ar@{-}[dl]\\
 & *{\circ}\ar@{-}[drr]_{} &  &  &  & *{\circ}\ar@{-}[dll]^{}\\
 &  &  & *{\bullet}\ar@{-}[d]_{}\\
 &  &  & *{\circ}\ar@{-}[d]\\
 &  &  & *{}\ar@{-}[]\\
 }
$~~~~~$\xyR{5pt}\xyC{5pt}\xymatrix{*{}\ar@{-}[d] &  & *{}\ar@{-}[d] &  & *{}\ar@{-}[dr] &  & *{}\ar@{-}[dl]\\
*{\bullet}\ar@{-}[dr] &  & *{\bullet}\ar@{-}[dl] &  &  & *{\circ}\ar@{-}[d]_{}\\
 & *{\circ}\ar@{-}[drr]_{} &  &  &  & *{\bullet}\ar@{-}[dll]^{}\\
 &  &  & *{\bullet}\ar@{-}[d]_{}\\
 &  &  & *{\circ}\ar@{-}[d]\\
 &  &  & *{}\\
 }
$

$\xyR{5pt}\xyC{5pt}\xymatrix{*{}\ar@{-}[dr] &  & *{}\ar@{-}[dl] &  & *{}\ar@{-}[d] &  & *{}\ar@{-}[d]\\
 & *{\circ}\ar@{-}[d]_{} &  &  & *{\bullet}\ar@{-}[dr] &  & *{\bullet}\ar@{-}[dl]\\
 & *{\bullet}\ar@{-}[drr]_{} &  &  &  & *{\circ}\ar@{-}[dll]^{}\\
 &  &  & *{\bullet}\ar@{-}[d]_{}\\
 &  &  & *{\circ}\ar@{-}[d]\\
 &  &  & *{}\\
 }
$~~~~~$\xyR{5pt}\xyC{5pt}\xymatrix{*{}\ar@{-}[dr] &  & *{}\ar@{-}[dl] &  & *{}\ar@{-}[dr] &  & *{}\ar@{-}[dl]\\
 & *{\circ}\ar@{-}[d]_{} &  &  &  & *{\circ}\ar@{-}[d]^{}\\
 & *{\bullet}\ar@{-}[drr]_{} &  &  &  & *{\bullet}\ar@{-}[dll]^{}\\
 &  &  & *{\bullet}\ar@{-}[d]_{}\\
 &  &  & *{\circ}\ar@{-}[d]\\
 &  &  & *{}\\
 }
$~~~~~$\xyR{5pt}\xyC{5pt}\xymatrix{*{}\ar@{-}[d] &  & *{}\ar@{-}[d] &  & *{}\ar@{-}[d] &  & *{}\ar@{-}[d]\\
*{\bullet}\ar@{-}[dr] &  & *{\bullet}\ar@{-}[dl] &  & *{\bullet}\ar@{-}[dr] &  & *{\bullet}\ar@{-}[dl]\\
 & *{\circ}\ar@{-}[d]_{} &  &  &  & *{\circ}\ar@{-}[d]_{}\\
 & *{\circ}\ar@{-}[drr] &  &  &  & *{\circ}\ar@{-}[dll]\\
 &  &  & *{\bullet}\ar@{-}[d]\\
 &  &  & *{}\\
 }
$

$\xyR{5pt}\xyC{5pt}\xymatrix{*{}\ar@{-}[d] &  & *{}\ar@{-}[d] &  & *{}\ar@{-}[dr] &  & *{}\ar@{-}[dl]\\
*{\bullet}\ar@{-}[dr] &  & *{\bullet}\ar@{-}[dl] &  &  & *{\circ}\ar@{-}[d]_{}\\
 & *{\circ}\ar@{-}[d]_{} &  &  &  & *{\bullet}\ar@{-}[d]\\
 & *{\circ}\ar@{-}[drr] &  &  &  & *{\circ}\ar@{-}[dll]\\
 &  &  & *{\bullet}\ar@{-}[d]\\
 &  &  & *{}\\
 }
$~~~~~$\xyR{5pt}\xyC{5pt}\xymatrix{*{}\ar@{-}[d] &  & *{}\ar@{-}[d] &  & *{}\ar@{-}[dr] &  & *{}\ar@{-}[dl]\\
*{\bullet}\ar@{-}[dr] &  & *{\bullet}\ar@{-}[dl] &  &  & *{\circ}\ar@{-}[d]_{}\\
 & *{\circ}\ar@{-}[d]_{} &  &  &  & *{\circ}\ar@{-}[d]\\
 & *{\circ}\ar@{-}[drr] &  &  &  & *{\bullet}\ar@{-}[dll]\\
 &  &  & *{\bullet}\ar@{-}[d]\\
 &  &  & *{}\\
 }
$~~~~~$\xyC{5pt}\xyR{5pt}\xymatrix{*{}\ar@{-}[dr] &  & *{}\ar@{-}[dl] &  & *{}\ar@{-}[d] &  & *{}\ar@{-}[d]\\
 & *{\circ}\ar@{-}[d] &  &  & *{\bullet}\ar@{-}[dr] &  & *{\bullet}\ar@{-}[dl]\\
 & *{\bullet}\ar@{-}[d] &  &  &  & *{\circ}\ar@{-}[d]\\
 & *{\circ}\ar@{-}[drr] &  &  &  & *{\circ}\ar@{-}[dll]\\
 &  &  & *{\bullet}\ar@{-}[d]\\
 &  &  & *{}\\
 }
$

$\xyC{5pt}\xyR{5pt}\xymatrix{*{}\ar@{-}[dr] &  & *{}\ar@{-}[dl] &  & *{}\ar@{-}[dr] &  & *{}\ar@{-}[dl]\\
 & *{\circ}\ar@{-}[d] &  &  &  & *{\circ}\ar@{-}[d]\\
 & *{\bullet}\ar@{-}[d] &  &  &  & *{\bullet}\ar@{-}[d]\\
 & *{\circ}\ar@{-}[drr] &  &  &  & *{\circ}\ar@{-}[dll]\\
 &  &  & *{\bullet}\ar@{-}[d]\\
 &  &  & *{}\\
 }
$~~~~~$\xyR{5pt}\xyC{5pt}\xymatrix{*{}\ar@{-}[dr] &  & *{}\ar@{-}[dl] &  & *{}\ar@{-}[dr] &  & *{}\ar@{-}[dl]\\
 & *{\circ}\ar@{-}[d]_{} &  &  &  & *{\circ}\ar@{-}[d]_{}\\
 & *{\bullet}\ar@{-}[d]_{} &  &  &  & *{\circ}\ar@{-}[d]\\
 & *{\circ}\ar@{-}[drr] &  &  &  & *{\bullet}\ar@{-}[dll]\\
 &  &  & *{\bullet}\ar@{-}[d]\\
 &  &  & *{}\\
 }
$~~~~~$\xyC{5pt}\xyR{5pt}\xymatrix{*{}\ar@{-}[dr] &  & *{}\ar@{-}[dl] &  & *{}\ar@{-}[d] &  & *{}\ar@{-}[d]\\
 & *{\circ}\ar@{-}[d]_{} &  &  & *{\bullet}\ar@{-}[dr] &  & *{\bullet}\ar@{-}[dl]\\
 & *{\circ}\ar@{-}[d] &  &  &  & *{\circ}\ar@{-}[d]_{}\\
 & *{\bullet}\ar@{-}[drr] &  &  &  & *{\circ}\ar@{-}[dll]\\
 &  &  & *{\bullet}\ar@{-}[d]\\
 &  &  & *{}\\
 }
$

$\xyC{5pt}\xyR{5pt}\xymatrix{*{}\ar@{-}[dr] &  & *{}\ar@{-}[dl] &  & *{}\ar@{-}[dr] &  & *{}\ar@{-}[dl]\\
 & *{\circ}\ar@{-}[d]_{} &  &  &  & *{\circ}\ar@{-}[d]_{}\\
 & *{\circ}\ar@{-}[d] &  &  &  & *{\bullet}\ar@{-}[d]_{}\\
 & *{\bullet}\ar@{-}[drr] &  &  &  & *{\circ}\ar@{-}[dll]\\
 &  &  & *{\bullet}\ar@{-}[d]\\
 &  &  & *{}\\
 }
$~~~~~$\xyC{5pt}\xyR{5pt}\xymatrix{*{}\ar@{-}[dr] &  & *{}\ar@{-}[dl] &  & *{}\ar@{-}[dr] &  & *{}\ar@{-}[dl]\\
 & *{\circ}\ar@{-}[d] &  &  &  & *{\circ}\ar@{-}[d]\\
 & *{\circ}\ar@{-}[d] &  &  &  & *{\circ}\ar@{-}[d]\\
 & *{\bullet}\ar@{-}[drr] &  &  &  & *{\bullet}\ar@{-}[dll]\\
 &  &  & *{\bullet}\ar@{-}[d]\\
 &  &  & *{}\\
 }
$
\end{center}
\end{example}

\begin{remark}
Notice that a planar structure on each of two trees $S$ and $T$ induces a planar structure on each of their shuffles, as illustrated by the way we have drawn the trees in the example above.
\end{remark}

%

\section{Characterizations of shuffles of trees}

In the examples presented in the previous section, one observes that each branch of a shuffle of two trees $S$ and $T$ is given by a classical shuffle of a branch of $S$ and one in $T$. We will start this section by showing that in fact this  provides an equivalent definition of the notion of shuffle.

\begin{prop}\label{ShuffleOfBranches}
Let $S$ and $T$ be two trees. A shuffle of $S$ and $T$ is a tree $A$ satisfying the conditions (1)-(3) of Definition~\ref{Def1} (on the labelling of the edges, the root and the leaves), as well as,
\begin{enumerate}
\item[(4$'$)] For any two leaves $s$ of $S$ and $t$ of $T$, the branch from the leaf $(s,t)$ of $A$ down to its root is a classical shuffle, of the two branches from $s$ down to the root in $S$ and from $t$ down to the root in $T$.
\end{enumerate}
\end{prop}

\begin{proof}
First note that for linear trees, both Definition~\ref{Def1} and Proposition~\ref{ShuffleOfBranches} are equivalent descriptions of classical shuffles.

For general trees, given the other conditions, it is then clear that Condition~(4) implies Condition~(4$'$) .

For the converse implication, let us consider an edge $(s,t)$ in $A$, where the edges immediately above $s$ are denoted $s_i$'s and the ones above $t$ are denoted $t_j$'s,
 as in the picture below Definition~\ref{Def1}.
Because of Condition~(4$'$), all edges above $(s,t)$ must be of the form $(s_i,t)$ or $(s,t_j)$. We have to show that no two different types occur, and all of the same type occur.
For the first property, suppose that some $(s_i,t)$ and some $(s,t_j)$  are successors of $(s,t)$. Let $\ell_i$ be a leaf above $s_i$ in $S$, and $\ell_j$ a leaf above $t_j$ in $T$. Then 
by Condition~(3), $(\ell_i,\ell_j)$ is in $A$. But then $A$ is not a treelike order.
For the second property, suppose now that one has only edges of type $(s_i,t)$ above $(s,t)$. Choose a leaf $\ell_i$ above each $s_i$, and a leaf $\ell$ above $t$. Then $(\ell_i, \ell)$ are all in $A$
by Condition~(3), so by Condition~(4$'$), all $(s_i,t)$ occur above $(s,t)$ (since as we have just seen no $(s,t_j)$ can).
\end{proof}

\begin{remark}\label{rem:Extension}
Conversely, given two trees $S$ and $T$, and two branches $b_S$ and $b_T$ inside them, it is not difficult to see that any classical shuffle of $b_S$ and $b_T$ can be extended to a shuffle of $S$ and $T$.
The proof is left to the reader.
\end{remark}

The set of edges of a shuffle $A$ is a subset of the cartesian product $E(S) \times E(T)$ of the set of edges of $S$ and of $T$, and the partial order induced by the tree structure of $A$ is the one induced by the partial orders on the edges of $S$ and $T$, respectively. These two partial orders are all treelike, of course. The notion of shuffle can also be described in these terms. This way of defining shuffles generalizes the classical notion of a shuffle of two linear trees $L$ and $M$ as a maximal linear suborder in $L \times M$, cf. Subsection~\ref{ss:ShufflesOfLinearTrees} above.

\begin{prop}\label{Maximality}
Let $S$ and $T$ be two trees. A shuffle of $S$ and $T$ is a subset $A \subset E(S) \times E(T)$ for which 
\begin{enumerate}
\item The partial order on $ E(S) \times E(T)$ restricts to a treelike partial order on $A$.
\item The largest elements in this partial order on $A$ are in bijective correspondence with the pairs $(s,t)$ where $s$ and $t$ range over the leaves of $S$ and $T$, respectively.
\item The set $A$ is maximal with these properties (1)-(2).
\end{enumerate}
\end{prop}

\begin{remark}
Conditions~(1) and~(2) together imply $A$ always contains $(r_S, r_T)$. To see the effect of Condition~(2), consider
the trees
\begin{equation*}
S:\vcenter{\xymatrix@R=10pt@C=14pt{
&&&\\
&*{ \circ} \ar@{-}[ul]^b\ar@{-}[ur]_c&& \\
&*{} \ar@{-}[u]_a \\
}}
 \ \text{ and }  \ \ 
T:\vcenter{\xymatrix@R=10pt@C=14pt{
&\\
&*{\bullet} \ar@{-}[u]_1 \\
&*{} \ar@{-}[u]_0 \\
}}\end{equation*}  
Then 
the tree $A$ pictured as 
\begin{equation*}
\vcenter{\xymatrix@R=10pt@C=14pt{
&&&\\
&&*{\ast}\ar@{-}[u]_{(c,1)}\\
&*{\ast} \ar@{-}[ul]^{(b,0)} \ar@{-}[ur]_{(a,1)}&& \\
&*{} \ar@{-}[u]_{(a,0)} \\
}}
\end{equation*}  
is maximal with Condition~(1), since adding either the edge $(b,1)$ or the edge $(c,0)$ will result in a partial order which is not treelike. This shows that Condition~(2) is necessary.
\end{remark}

\begin{proof}
Suppose $A$ is a shuffle in the sense of Proposition~\ref{ShuffleOfBranches}. Then, by definition, the partial order $A$ is treelike and the pairs of leaves are exactly the largest elements of $A$.
Only the maximality remains to be shown. Recall that Condition (4$'$) implies that for every pair of leaves $(\ell,\ell')$ the branch $\downarrow (\ell,\ell')$ is a maximal linear order, that 
is $\downarrow (\ell,\ell')$ is of the form $\{(r_S,r_T) <a_1 < \ldots <a_n < (\ell, \ell') \}$ such that there is no $b \in E(S) \times E(T)$ with $a_i<b<a_{i+1}$.
Suppose now there exists some edge $(s,t)$ not in $A$, such that $A \cup \{(s,t)\}$ is a tree. Consider a leaf $(\ell,\ell')$ above $(s,t)$. Then  $\downarrow (\ell,\ell')$ (in $A \cup \{(s,t)\}$) 
is the set $\{(r_S,r_T) ,a_1 , \ldots ,a_n , (\ell, \ell'), (s,t)\}$ with a total order, which is impossible.

Conversely, consider a subset $A \subset E(S) \times E(T)$, maximal with properties (1) and (2). Let us show it is a shuffle in the sense of Definition~\ref{Def1}.
All conditions are obvious except the last one. We first prove this condition for the root, and then proceed by induction.

Suppose $S$ has edges $s_1, \ldots, s_m$ above its root edge $r_S$ and $T$ has $t_1, \ldots, t_n$ above its root edge $r_T$. 
Write $S_i=\{s \in S \, ; \, s \geq s_i\}$ and $T_j=\{t \in T \, ; \, t \geq t_j\}$ for the corresponding subtrees.
We know $(r_S, r_T)$ is in $A$. Suppose that some $(s_i, r_T)$ is not in $A$. The maximality condition implies that $A \cup \{(s_i,r_T)\}$ is not a tree, that is we can have 
either two incomparable elements $a$ and $b$ in $A$ below $(s_i, r_T)$, or $a=(a_S, a_T)$ and $b=(b_S, b_T)$ in $A$ with $(s_i, r_T)<a$ and $b<a$ and $(s_i, r_T)$ not comparable with $b$. 
The first option is impossible: 
consider a leaf $(\ell, \ell')$ above $(s_i, r_T)$; this leaf is in $A$ (as are all the leaves) and is above the incomparable elements $a$ and $b$, which contradicts $A$ being a tree.
The second option implies that $a_S$ is in $S_i$, and thus $b_S$ is in $S_i$ or is the root of $S$. That means $b$ and $(s_i, r_T)$ can be incomparable only if $b_S$ is the root of $S$ (and then $b_T$ is not the root of $T$).
This proves that for each $i$, either $(s_i, r_T)$ is in $A$ or some $(r_S,t)$ (with $t \neq r_T$) is in $A$.
Symmetrically, for each $j$, either $(r_S, t_j)$ is in $A$ or some $(s,r_T)$ (with $s \neq r_S$) is in $A$.
But observe that two elements $(r_S,t)$ (with $t \neq r_T$) and $(s,r_T)$ (with $s \neq r_S$) can never be simultaneously in $A$ (as they have a common leaf above them).
Thus we get that either for each $i$, $(s_i, r_T)$ is in $A$, or  for each $j$, $(r_S, t_j)$ is in $A$.

We continue by induction. Suppose for instance, that for each $i$, $(s_i, r_T)$ is in $A$.
Let $A_i$ be the intersection of $A$ with $S_i \times T$. It suffices to prove that each $A_i$ is a maximal tree in $S_i \times T$.
If it is not maximal, we can find $x=(x_S,s_T)$ not in $A_i$ such that $A_i \cup \{x\}$ is a tree. But by the maximality condition on $A$, we know that $A \cup \{x\}$ is not a tree.
So, as before, we find elements $a$ and $b$ in $A$ such that $x<a$, $b<a$ and $x$ and $b$ incomparable. Since $x$ lies in $S_i \times T$, $a$ is in $A_i$.
Now, as $b<a$, we obtain that $b$ is either the root or lies in $A_i$. 
But $b$ cannot be the root as the root is comparable with any element. Moreover $b$ cannot be in $A_i$ because $A_i \cup \{x\}$ is a tree, which contradicts $x$ and $b$ being incomparable 
while being both below $a$.
This contradiction shows that each $A_i$ is maximal.
\end{proof}

We will conclude this section with an inductive characterization of the notion of a shuffle of two trees.
To this end, notice that a tree $S$ is either the unit tree $\eta$, or is obtained by grafting $m>0$ trees $S_1, \ldots, S_m$ onto a corolla $C_m$; in this case we will write $S=C_m[S_1, \ldots, S_m]$.

Here is a picture of a corolla $C_m$ and $S=C_m[S_1, \ldots, S_m]$.
\begin{equation*}
\vcenter{\xymatrix@R=10pt@C=14pt{
&&\\
&*{ \circ}\ar@{}[u]|{\cdots} \ar@{-}[ul]^{1}\ar@{-}[ur]_{m}&& \\
&*{} \ar@{-}[u]_{0} \\
}}
 \ \  \ \ 
\vcenter{\xymatrix@R=10pt@C=10pt{
*{}\ar@{-}[rr]&&*{}&&*{}\ar@{-}[rr]&&*{}\\
&S_1&&\cdots\cdots&&S_m&&\\
&*{}\ar@{-}[uul]\ar@{-}[uur]&&&&*{}\ar@{-}[uul]\ar@{-}[uur]\\
&&&*{ \bullet}\ar@{}[u]|{\cdots\cdots} \ar@{-}[ull]\ar@{-}[urr]&& \\
&&&*{} \ar@{-}[u] \\
}}\end{equation*}  

\begin{prop}\label{InductionProperty}
The set of shuffles $Sh(S,T)$ of two trees $S$ and $T$ satisfies the following three properties:
\begin{enumerate}
\item (symmetry) $Sh(S,T) \simeq Sh(T,S)$.
\item (unit) $Sh(S, \eta)$ and $Sh(\eta,T)$ are one-point sets.
\item (induction) If $S= C_m[S_1, \ldots, S_m]$ and $T=C_n[T_1, \ldots, T_n]$, then there is a canonical bijection
$$\theta : Sh(S,T) \to \prod_{i=1}^m Sh(S_i,T) + \prod_{j=1}^{n} Sh(S,T_j).$$
(Here $\prod$ and $+$ denote cartesian product and disjoint sum of sets.)
\end{enumerate}
\end{prop}

\begin{proof}
The first and second properties are obvious. The last one follows from the description in Condition (3) of Definition~\ref{Def1}.
\end{proof}

\begin{remark}
Of course, these properties determine the set of shuffles uniquely up to isomorphism, in the sense that
if $F$ is an operation associating a set $F(S,T)$ to any pair of trees and $\tau$ is a bijection
$$\tau : F(S,T) \to \prod_{i=1}^m F(S_i,T) + \prod_{j=1}^{n} F(S,T_j),$$
such that $F(S, \eta)$ and $F(\eta,T)$ are singletons, then
\begin{itemize}
 \item there is a natural bijection $F(S,T) \to F(T,S)$
 \item there is a bijection $F(S,T) \to Sh(S,T)$ compatible with $\tau$ and $\theta$.
\end{itemize}
\end{remark}

\begin{remark}
From this description, we see that the set $Sh(S,T)$ carries a natural partial order, again defined by induction:
if the partial orders on shuffles of smaller trees have been defined, we order $Sh(S,T)$ by taking the product partial orders on 
$\prod Sh(S_i,T)$ and  $\prod Sh(S,T_j)$, and then ordering the two summands by declaring elements in the left hand summand to be smaller than those in the right hand one. The symmetry isomorphism in (3) reverses this order. We will discuss this partial order in more detail in Section~\ref{S:Lattice}.
\end{remark}

\begin{remark}\label{Rem:Stumps}
In the literature (cf. \cite{HHM}, for example) one also considers shuffles of trees with stumps (also referred to as bald vertices or vertices of valence zero).
From our point of view, these can be described as follows:
for two trees $S$ and $T$, possibly with stumps, let $S^0$ and $T^0$ be the subtrees obtained by pruning the stumps away. 
Then the shuffles of $S$ and $T$ are in bijective correspondence with those of $S^0$ and $T^0$. In fact, a shuffle of the former two as considered in loc. cit.
is obtained from a shuffle $A$ of the latter two by reintroducing  a stump above each pair of leaves of $S^0$ and $T^0$  which aren't both leaves in $S$ and in $T$, respectively.

Here is an example, with the three shuffles obtained from the following two trees:
$$
S:\vcenter{\xymatrix@R=8pt@C=10pt{
&&*{ \circ} \\
 &*{ \circ} \ar@{-}[ul]^2\ar@{-}[ur]_3&&\\
&*{ \circ} \ar@{-}[u]^1&& \\
&*{} \ar@{-}[u]^0 \\
}}
 \  \text{ and  }   \quad
T:\vcenter{\xymatrix@R=8pt@C=10pt{
& &&*{ \bullet}&&\\
&&*{ \bullet} \ar@{-}[ul]^b\ar@{-}[ur]_c&& \\
&&*{} \ar@{-}[u]^a \\
}}
$$

$$\vcenter{\xymatrix@R=9pt@C=9pt{
&&*{ \circ} &&*{ \bullet}  && *{\odot} &\\
&*{ \circ} \ar@{-}[ul]^{b2}\ar@{-}[ur]_{b3}&&&&*{ \circ} \ar@{-}[ul]^{c2}\ar@{-}[ur]_{c3}&&\\
&*{ \circ} \ar@{-}[u]^{b1} &&&&*{ \circ} \ar@{-}[u]_{c1}&&\\
&&&*{ \bullet} \ar@{-}[ull]^{b0}\ar@{-}[urr]_{c0}&& \\
&&&*{} \ar@{-}[u]^{a0} \\
}}
  \text{,  }  
\vcenter{\xymatrix@R=8pt@C=8pt{
&&*{ \circ} &&*{ \bullet}  && *{\odot} &\\
&*{ \circ} \ar@{-}[ul]^{b2}\ar@{-}[ur]_{b3}&&&&*{ \circ} \ar@{-}[ul]^{c2}\ar@{-}[ur]_{c3}&&\\
&&&*{ \bullet} \ar@{-}[ull]^{b1}\ar@{-}[urr]_{c1}&& \\
&&&*{ \circ} \ar@{-}[u]^{a1}&& \\
&&&*{} \ar@{-}[u]^{a0} \\
}}
  \text{and  } 
\vcenter{\xymatrix@R=8pt@C=8pt{
&&*{ \bullet} &&*{ \circ}  && *{\odot} &\\
&*{ \bullet} \ar@{-}[ul]^{b2}\ar@{-}[ur]_{c2}&&&&*{ \bullet} \ar@{-}[ul]^{b3}\ar@{-}[ur]_{c3}&&\\
&&&*{ \circ} \ar@{-}[ull]^{a2}\ar@{-}[urr]_{a3}&& \\
&&&*{ \circ} \ar@{-}[u]^{a1}&& \\
&&&*{} \ar@{-}[u]^{a0} \\
}}
$$
\end{remark}

%

\section{The number of shuffles}\label{S:Number}

For two trees $S$ and $T$, we will write $sh(S,T)$ for the number of shuffles of $S$ and $T$, in contrast with the set of shuffles $Sh(S,T)$. 
This number grows very fast in the size of the trees $S$ and $T$, and it seems difficult to give some uniform closed formulas giving precise information about this number as a function of $S$ and $T$, except for some very special types of trees. In this section, we will present some simple observations concerning the number $sh(S,T)$. To begin with, from Proposition~\ref{InductionProperty} one immediately obtains the following.

\begin{prop}
The set of shuffles $Sh(S,T)$ of two trees $S$ and $T$ satisfies the following three properties:
\begin{enumerate}
\item $sh(S,T)=sh(T,S)$.
\item If $T$ is the unit tree $\eta$, $sh(S,\eta)=1$.
\item If $S= C_m[S_1, \ldots, S_m]$ and $T=C_n[T_1, \ldots, T_n]$, then
$$sh(S,T)= \prod_{i=1}^m sh(S_i,T) + \prod_{j=1}^{n} sh(S,T_j).$$
\end{enumerate}
\end{prop}

\begin{remark}
For the linear  trees $L_p$ and $L_q$ with $p$ and $q$ vertices respectively, let us write $\lambda(p,q)$ for $sh(L_p,L_q)$.
Then the proposition states that $\lambda(p,0)=1=\lambda(0,q)$ and $\lambda(p,q)=\lambda(p-1,q)+\lambda(p,q-1)$, which is the familiar inductive relation defining the binomial coefficient $\binom{p+q}{p}=\binom{p-1+q}{p-1}+\binom{p+q-1}{p}$. It corresponds to the generating function defined by $\Lambda(x,y)=(x+y) \Lambda(x,y)$ for which $\Lambda(x,y)= \sum \lambda(p,q)x^p y^q$.

A similar simple relation already fails for binary trees. For example, let us write $B_p$ for the binary tree all of whose branches have $p$ vertices,
\begin{equation*}
B_1:\vcenter{\xymatrix@R=8pt@C=9pt{
&&\\
&*{ \circ} \ar@{-}[ul]\ar@{-}[ur]&&& \\
&*{} \ar@{-}[u] \\
}}
  \text{   }  
B_2:\vcenter{\xymatrix@R=8pt@C=9pt{
&&&&\\
&*{ \circ} \ar@{-}[ul]\ar@{-}[ur]&&*{ \circ} \ar@{-}[ul]\ar@{-}[ur]&&&\\
&&*{ \circ} \ar@{-}[ul]\ar@{-}[ur]&& \\
&&*{} \ar@{-}[u] \\
}}
  \text{   }   
B_3:\vcenter{\xymatrix@R=8pt@C=9pt{
&&&&&&&&\\
&*{ \circ} \ar@{-}[ul]\ar@{-}[ur]&&*{ \circ} \ar@{-}[ul]\ar@{-}[ur]&&*{ \circ} \ar@{-}[ul]\ar@{-}[ur]&&*{ \circ} \ar@{-}[ul]\ar@{-}[ur]&\\
&&*{ \circ} \ar@{-}[ul]\ar@{-}[ur]&&&&*{ \circ} \ar@{-}[ul]\ar@{-}[ur]&&\\
&&&&*{ \circ} \ar@{-}[ull]\ar@{-}[urr]&& \\
&&&&*{} \ar@{-}[u] \\
}}\end{equation*}

\noindent and $\beta(p,q)$ for $sh(B_p,B_q)$. Then the proposition gives $\beta(p,0)=1=\beta(0,q)$ and for $p,q>0$,
$$\beta(p,q)=\beta(p-1,q)^2 +\beta(p,q-1)^2.$$
But the coefficients of the generating function $F$ defined by $F(x,y)= (x+y)F(x,y)^2$ do not describe these shuffles. Instead, the coefficient of $x^p y^q$ describes the number of binary trees with $p$ white vertices and $q$ black ones, a related but different combinatorial number.
\end{remark}

One can describe upper and lower bounds for the number $sh(S,T)$. 
Let us define the \textit{height} $ht(S)$ of a tree $S$ as the number of vertices on the longest branch.

\begin{prop}
For any two trees $S$ and $T$, the following inequalities hold:
$$\binom{ht(S)+ht(T)}{ht(S)} \leq sh(S,T) \leq \prod_{\alpha,\beta} \binom{|\alpha|-1+|\beta|}{|\beta|} + \prod_{\alpha,\beta} \binom{|\alpha|+|\beta|-1}{|\alpha|}$$
$$ \leq \prod_{\alpha,\beta} \binom{|\alpha|+|\beta|}{|\beta|}$$
where $\alpha$ and $\beta$ range over all branches of $S$ and of $T$, respectively, and $|\alpha|$ denotes the height of a branch.
\end{prop}

\begin{proof}
If $\alpha$ and $\beta$ are the longest branches of $S$ and $T$, we have observed in Remark~\ref{rem:Extension} that a shuffle of $\alpha$ and $\beta$ can be extended to a shuffle of all of $S$ and $T$. 
This gives the lower bound.
Since an arbitrary shuffle of $S$ and $T$ is determined by the family of shuffles of their branches, we obtain the coarser upper bound. The sharper one simply comes from the observation that any shuffle must start with the root of $S$ or with the root of $T$. 
\end{proof}

\begin{example}\label{ExCranch1}
The lower bound is in general not very sharp. Indeed, in the case of Example~\ref{ExampleWith14}, with both trees of height $2$, we already counted $14$ shuffles, while the binomial coefficient of the lower bound is $6$.
A computer calculation by James Cranch shows that for the trees $S$ and $T$

\begin{equation*}
S:\vcenter{\xymatrix@R=10pt@C=14pt{
&&&&\\
&*{ \circ} \ar@{-}[u]&*{ \circ} \ar@{-}[u]&*{ \circ} \ar@{-}[u]&&\\
&&*{ \circ} \ar@{-}[ul]\ar@{-}[u]\ar@{-}[ur]&& \\
&&*{ \circ } \ar@{-}[u] \\
&&*{} \ar@{-}[u] \\
}}
\ \text { and } \
T:\vcenter{\xymatrix@R=10pt@C=14pt{
&&&&&&\\
&*{ \bullet} \ar@{-}[u]&&&&\\
&*{ \bullet} \ar@{-}[u]&&&&\\
&*{ \bullet} \ar@{-}[u]&&*{ \bullet} \ar@{-}[u]&&\\
&&*{ \bullet} \ar@{-}[ul]\ar@{-}[ur]&& \\
&&*{} \ar@{-}[u] \\
}}
\end{equation*}

\noindent of height $3$ and $4$ respectively, the number of shuffles is 3089. The fact that this is a prime number perhaps shows the lack of symmetry and additivity properties of $sh(S,T)$.
\end{example}

As a final illustration of the growth of the number of $sh(S,T)$, let us consider shuffles of a tree $S$ with linear trees. For such a fixed tree $S$, define an integer function $P_S: \mathbb N \to \mathbb N$ by $P_S(n)=sh(S,L_n)$.
Thus, for example, $P_S(0)=1$ for any tree $S$, while $P_S(n)=1$ if $S$ is the unit tree $\eta$ and $P_S(n)=n+1$ if $S$ has one vertex, i.e. is a corolla. For the proposition, let us introduce some notation. If $S$ is a tree and $v$ is a vertex in $S$, write $S_v=\{s \in S \, ; \, s \geq v\}$ for the maximal subtree of $S$ with $v$ as its root vertex,
$$\xymatrix@R=10pt@C=14pt{
&&&&&\\
&&*{S_v} &&\\
&&*{\bullet} \ar@{-}[uul] \ar@{-}[uur]\ar@{}[r]_{v \quad }\ar@{-}[d] &&\\
&&&&\\
&&*{ \bullet} \ar@{-}[uuuull]\ar@{-}[uuuurr]&& \\
&&*{} \ar@{-}[u] \\
}$$

and write $S!$ for the number 
$$S!=\displaystyle {\prod_{v \in S} |S_v|}$$
where $v$ ranges over all the vertices in $S$, and $|S_v|$ is the number of vertices in $S_v$.
Then for a linear tree $L_n$ one has $L_n!=n!$ and for $S=C_m[S_1, \ldots, S_m]$ one has
$$S! \, = \, |S| \left(\prod_{i=1}^m S_i !\right).$$

\begin{prop}
For any tree $S$ the function $P_S: \mathbb N \to \mathbb N$ is a polynomial of degree $|S|$ with rational coefficients, and leading coefficient 
$(S!)^{-1}$.
\end{prop}

\begin{proof}
The proposition is obviously true for a corolla $S$, as already pointed out.
For a larger tree $S=C_m[S_1, \ldots, S_m]$, we can use Proposition~\ref{InductionProperty}(3) to write
$$P_S(n)= \prod_{i=1}^m P_{S_i}(n) + P_S(n-1)$$
$$=\prod_{i=1}^m P_{S_i}(n) + \prod_{i=1}^m P_{S_i}(n-1) + \ldots + \prod_{i=1}^m P_{S_i}(0).$$
If we write the leading term of $P_S(n)$ as $c_S n^{d_S}$, and that of $P_{S_i}(n)$ as $c_i n^{d_i}$, 
then $c_S n^{d_S}$ is also the leading term of 
$$\prod_{i=1}^m c_i n^{d_i} + \prod_{i=1}^m c_i (n-1)^{d_i} + \ldots + \prod_{i=1}^m c_i 0^{d_i}$$
$$= \left(\prod_{i=1}^m c_i \right) \left(\sum_{j=0}^n j^d\right)$$
where $d=d_1+ \ldots + d_m$. But by Faulhaber's formula \cite{Knuth},
$\sum_{j=0}^n j^d$ is a polynomial  in $n$ with rational coefficients, of degree $d+1$ and with leading coefficient $(d+1)^{-1}$. So we find that 
$$c_S= (\prod_{i=1}^m c_i ) (d+1)^{-1} \ \ \text{ and } \ \ d_S=d+1. $$
But if $d_i=|S_i|$ then $d+1=|S|$, and if $c_i=(S_i!)^{-1}$ then $c_S=(S!)^{-1}$. So the proposition is proved.
\end{proof}

\begin{example}\label{ExCranch2}
The polynomial $P_S$ does not determine $S$, as the following example, pointed out to us by James Cranch, shows.
Consider the following two trees $S$ and $R$,

\begin{equation*}
S:\vcenter{\xymatrix@R=10pt@C=14pt{
&&&&\\
&*{ \circ} \ar@{-}[u]&*{ \circ} \ar@{-}[u]&*{ \circ} \ar@{-}[u]&&\\
&&*{ \circ} \ar@{-}[ul]\ar@{-}[u]\ar@{-}[ur]&& \\
&&*{ \circ } \ar@{-}[u] \\
&&*{} \ar@{-}[u] \\
}}
\ \text { and } \
R:\vcenter{\xymatrix@R=10pt@C=14pt{
&&&&&&\\
&*{ \bullet} \ar@{-}[u]&&*{ \bullet} \ar@{-}[u]&&\\
&*{ \bullet} \ar@{-}[u]&&*{ \bullet} \ar@{-}[u]&&\\
&&*{ \bullet} \ar@{-}[ul]\ar@{-}[ur]&& \\
&&*{} \ar@{-}[u] \\
}}
\end{equation*}

\noindent
and let us calculate $P_S(n)$ and $P_T(n)$. A shuffle of $S$ and the linear tree $L_n$ is described by putting $i$ vertices of $L_n$ below the root of $S$, then $j$ vertices  between the bottom two vertices, so that on each of the three branches, one is left with a shuffle of the top $n-i-j$ vertices of $L_n$ and the one top vertex of $S$ on that branch.
There are $n-i-j+1$ of the latter shuffles on each branch, so
$$P_S(n)= \sum_{i=0}^n \sum_{j=0}^{n-i} (n-i-j+1)^3 
=\sum_{k=0}^n \sum_{j=0}^{k} (k-j+1)^3
=\sum_{k=0}^n \sum_{l=1}^{k+1} l^3
=\sum_{k=0}^n \binom{k+2}{2}^2$$
where we have used the familiar formula for the sum of cubes.

On the other hand, a shuffle of $R$ and $L_n$ is described by putting $i$ vertices of $L_n$ below the root of $R$, and then shuffling each of the two linear branches of $R$ with the remaining $(n-i)$ vertices of $L_n$. So
$$P_R(n)= \sum_{i=0}^n \binom{n-i+2}{2}^2 =\sum_{k=0}^n \binom{k+2}{2}^2$$
also, showing $P_R=P_S$.
\end{example}

%

\section{The lattice of shuffles}\label{S:Lattice}


The inductive description of the set of shuffles of two trees $S$ and $T$ given at the end of Section~\ref{S:Number},
$$Sh(S,T)= \prod_{i=1}^m Sh(S_i,T) + \prod_{j=1}^{n} Sh(S,T_j),$$
defines a partial order on the set of shuffles, by declaring elements in the left summand to be smaller than those in the right and next proceeding by induction. We would like to describe a more picturesque way of introducing this partial order. Let us draw the vertices of $S$ as white and those of $T$ as black. Then the first (smallest) element in $Sh(S,T)$ is the one with a copy of the black tree $T$ on top of each of the leaves of $S$. The elements in the partial order immediately above that are obtained by ``percolating'' the copies of the black root vertex of $T$ just above a single top white vertex of $S$ through that white vertex. If the root of $T$ has $n$ incoming edges and that particular top vertex of $S$ has $m$ incoming edges then this involves replacing $m$ copies of the root of $T$ by just one, and replacing the top vertex involved of $S$ by $n$ copies of it.
Here is a picture for $n=2$ and $m=3$:

\includegraphics[scale=0.8]{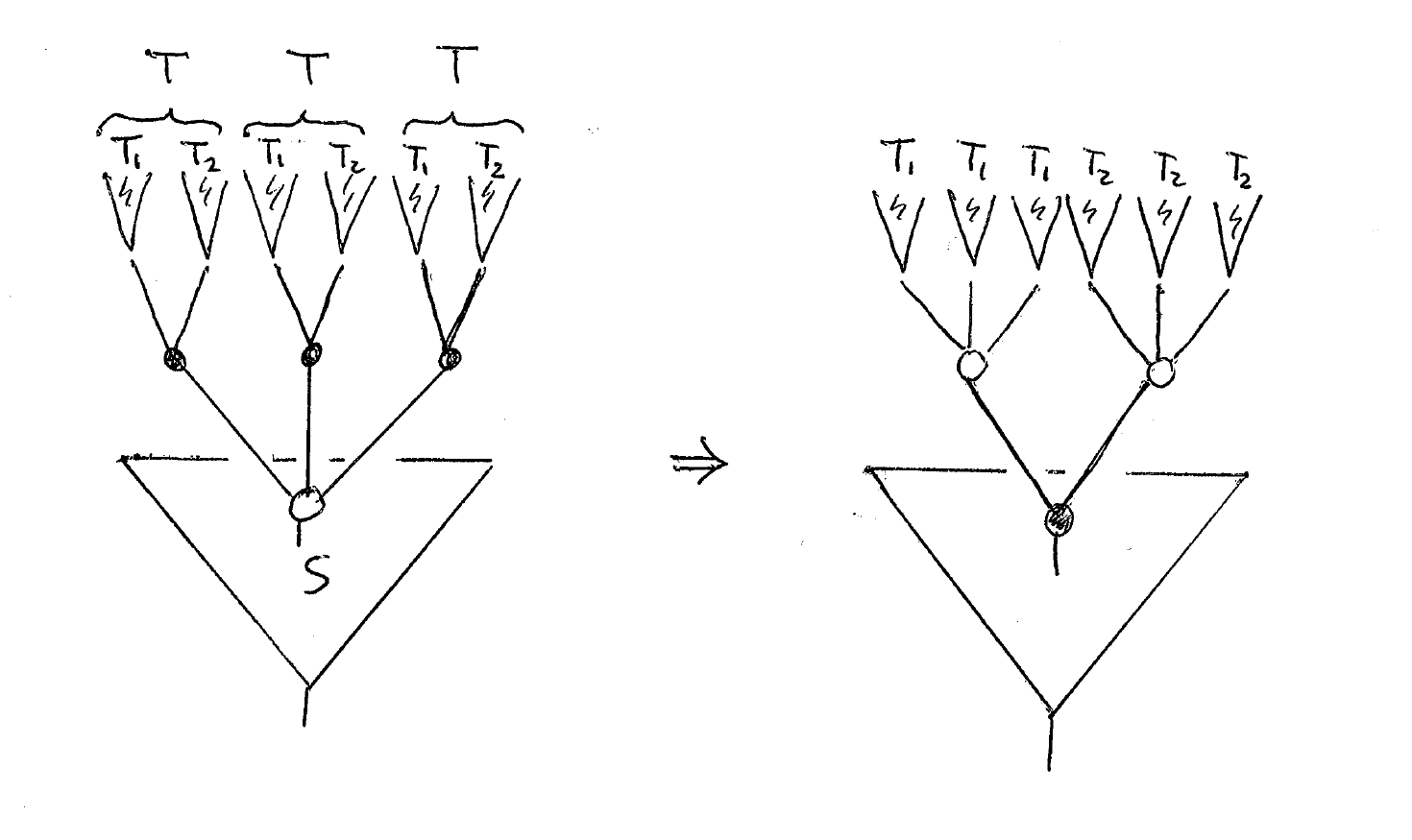}

If, more generally, we change a shuffle by locally replacing a white vertex with $m$ incoming edges and $m$ copies of the same $n$-ary black vertex on top, by percolating each black vertex through the white one to obtain one single black vertex with $n$ copies of the white one on top, we obtain a new shuffle of which we will say that it is obtained from the first one by a \textit{single percolation step}. Below is a picture illustrating such a single percolation step. Notice that the figures below have the same ``boundary'', i.e. the same root and the same set of local leaves, be it in a different planar order. Therefore such a percolation step can indeed be performed locally, leaving the rest of the shuffle untouched. 

\begin{equation*}
\vcenter{\xymatrix@R=10pt@C=14pt{
&&&&&&&\\
&*{ \bullet}\ar@{}[u]|{\cdots} \ar@{-}[ul]^{(s_1,t_1)}\ar@{-}[ur]_{(s_1,t_n)}&&&&*{ \bullet}\ar@{}[u]|{\cdots} \ar@{-}[ul]^{(s_m,t_1)}\ar@{-}[ur]_{(s_m,t_n)}&&\\
&&&*{ \circ}\ar@{}[u]|{\cdots\cdots} \ar@{-}[ull]^{(s_1,t)}\ar@{-}[urr]_{(s_m,t)}&& \\
&&&*{} \ar@{-}[u]_{(s,t)} \\
}}
\quad \Longrightarrow \quad
\vcenter{\xymatrix@R=10pt@C=14pt{
&&&&&&&\\
&*{ \circ}\ar@{}[u]|{\cdots} \ar@{-}[ul]^{(s_1,t_1)}\ar@{-}[ur]_{(s_m,t_1)}&&&&*{ \circ}\ar@{}[u]|{\cdots} \ar@{-}[ul]^{(s_1,t_n)}\ar@{-}[ur]_{(s_m,t_n)}&&\\
&&&*{ \bullet}\ar@{}[u]|{\cdots\cdots} \ar@{-}[ull]^{(s,t_1)}\ar@{-}[urr]_{(s,t_n)}&& \\
&&&*{} \ar@{-}[u]_{(s,t)} \\
}}
\end{equation*}  
Above is a picture of a percolation step locally, for a part of $S$ and a part of $T$ drawn below.
\begin{equation*}
\vcenter{\xymatrix@R=10pt@C=14pt{
&&\\
&*{ \circ}\ar@{}[u]|{\cdots} \ar@{-}[ul]^{s_1}\ar@{-}[ur]_{s_m}&& \\
&*{} \ar@{-}[u]_{s} \\
}}
 \ \  \ \ 
\vcenter{\xymatrix@R=10pt@C=14pt{
&&\\
&*{ \bullet}\ar@{}[u]|{\cdots} \ar@{-}[ul]^{t_1}\ar@{-}[ur]_{t_n}&& \\
&*{} \ar@{-}[u]_{t} \\
}}\end{equation*}

With this definition, a shuffle $A$ comes before a shuffle $B$ in the partial order on $Sh(S,T)$ precisely when $B$ can be obtained from $A$ by a sequence of percolation steps.
If we start with the minimal shuffle having copies of the black tree $T$ on top of a single white tree $S$, and perform all possible sequences of percolation steps, then we obtain all the possible shuffles of $S$ and $T$, ending with the maximal one having copies of $S$ on top of the tree $T$.

\bigskip

For a tree $T$, recall that we write $V(T)$ for the partial order of vertices of $T$, taking the root vertex as smallest element. 
From now on, $r_T$ will now denote the root vertex rather than the root edge.
We will also have occasion to use the opposite partial order, denoted $V(T)^\vee$.
Any partial order $P$ defines a topology on its set of elements, the so-called \textit{Alexandrov topology}, the open sets of which are the upwards closed sets. Explicitly, $U \subset P$ is open if and only if
$$( p \leq q \text{ and } p \in U ) \Longrightarrow q \in U$$
for any $p,q$ in $P$. Every such open set is the union of basic open sets of the form 
$$B(p)=\{q \in P \, | \, p \leq q \}.$$
Let us write $\mathcal O(P)$ for the lattice of open subsets of $P$. With this notation, we can state the following result.

\begin{prop}
Let $S$ and $T$ be two trees. There is a natural isomorphism of partially ordered sets 
$$\phi: Sh(S,T) \to \mathcal O (V(S) \times V(T)^\vee).$$
\end{prop}

\begin{remark}
An intuitive way to describe this isomorphism is as follows: Thinking of a shuffle $A$ of $S$ and $T$ as obtained by percolation steps from the minimal shuffle (copies of the black tree $T$ on top of the white one $S$), the corresponding open set consists of exactly the pairs of vertices $(v,w) \in V(S) \times V(T)$ for which $w$ has percolated through $v$ in $A$.
For example, in Example~\ref{ExampleWith14}, the first tree consists in copies of $T$ above $S$, and the corresponding open set of percolated vertices is empty. 
To obtain the second tree, the root vertex of $T$ has percolated through the top vertex of $S$, and thus the corresponding open set is just the singleton consisting of the pair with the top vertex of $S$ and the root vertex of $T$.
\end{remark}

\begin{proof}
Note that if $S$ or $T$ is the unit tree $\eta$ with just a single edge, then $Sh(S,T)$ is a single point, as is the lattice of open subsets of the empty partial order, so the proposition is still valid.

We first show that $\mathcal O (V(S) \times V(T)^\vee)$ allows a similar inductive description as $Sh(S,T)$. To this end, write
$S=C_m[S_1, \ldots, S_m]$ and $T=C_n[T_1, \ldots, T_n]$, as before. Consider an open set $U \subset V(S) \times V(T)^\vee$. 
If the pair $(r_S, r_T)$ of root vertices does not belong to $U$, then no $(r_S,t)$ can belong to $U$, so $U$ is in fact a subset of
$\coprod V(S_i) \times V(T)^\vee$. On the other hand, if $(r_S, r_T)$ does belong to $U$, then $S \times \{r_T\} \subset U$, so $U$ is of the form
$S \times \{r_T\} \cup U_1 \cup \ldots \cup U_n$ where $U_j \subset V(S) \times V(T_j)^\vee$ are disjoint open sets.
These $U_j$'s together with the fact that  $(r_S, r_T) \in U$  determine $U$. These observations define an isomorphism:
$$\tau : \mathcal O (V(S) \times V(T)^\vee) \, \to \, \prod \mathcal O (V(S_i) \times V(T)^\vee) + \prod \mathcal O (V(S) \times V(T_j)^\vee).$$
We can now simply define the isomorphism $\phi=\phi_{S,T}$ of the proposition by induction, in a way that respects the two ``induction isomorphisms'' $\theta$ and $\tau$, as in the diagram below.

\begin{equation*}
\xymatrix@R=30pt@C=10pt@M=0pt{
Sh(S,T) \ar[d]_{\varphi_{S,T}}  & \ar[rrrrr]_\theta^{\sim}&&&&&&
\prod_{i=1}^m Sh(S_i,T) + \prod_{j=1}^{n} Sh(S,T_j) \ar[d]^{\prod\varphi_{S_i,T} + \prod\varphi_{S,T_j}}\\
\mathcal O (V(S) \times V(T)^\vee)& \ar[rrrrr]_\tau^{\sim}&&&&&&
  \prod \mathcal O (V(S_i) \times V(T)^\vee) + \prod \mathcal O (V(S) \times V(T_j)^\vee)
}
\end{equation*}

\end{proof}

\begin{remark}
Since open sets in a poset correspond to antichains of elements in that poset, the problem of counting shuffles is a special case of that of counting such antichains, a problem known to be notoriously difficult.
\end{remark}

\begin{cor}
The partial order $Sh(S,T)$ is a finite distributive lattice.
\end{cor}

Recall the notion of a reduced tree from Section~\ref{S:Def}.
Any tree $T$ has a unique reduction $T_{red}$, obtained by pruning away all the leaves, except for leaving exactly one leaf at each top vertex.

\begin{equation*}
T:\vcenter{\xymatrix@R=8pt@C=10pt{
&&&&&\\
&*{ \bullet} \ar@{-}[ul]\ar@{-}[u]\ar@{-}[ur]\ar@{-}[urr]&&&\\
&*{ \bullet} \ar@{-}[ul]\ar@{-}[u]\ar@{-}[ur]&&*{ \bullet} \ar@{-}[ul]\ar@{-}[ur]&& \\
&&*{ \bullet} \ar@{-}[ul]\ar@{-}[u]\ar@{-}[ur]&& \\
&&*{} \ar@{-}[u] \\
}}
\ \text { and } \
T_{red}:\vcenter{\xymatrix@R=7pt@C=8pt{
&&&&&&\\
&*{ \bullet} \ar@{-}[u]&&\\
&*{ \bullet} \ar@{-}[u]&&*{ \bullet} \ar@{-}[u]&&\\
&&*{ \bullet} \ar@{-}[ul]\ar@{-}[ur]&& \\
&&*{} \ar@{-}[u] \\
}}
\end{equation*}

This reduced tree $T_{red}$ is the minimal tree for which  $V(T_{red}) \simeq V(T)$. The previous proposition clearly implies that for the study of shuffles, one can restrict oneself to reduced trees. We state this explicitly as:

\begin{cor}
For any two trees $S$ and $T$ there is a canonical isomorphism of partially ordered sets:
$$Sh(S,T) = Sh(S_{red},T_{red}).$$
\end{cor}

Before stating the next corollary, let us introduce some more notation. If $P$ is a partial order, each open set $U \subset P$ has a characteristic function $\chi_U : P \to L_2$ where $L_2$ is the linear order with two vertices $v$ and $w$
\begin{equation*}
L_2: \vcenter{\xymatrix@R=7pt@C=8pt{
\\
*{\ \ \circ \, v} \ar@{-}[u] \\
*{\ \ \, \circ \, w} \ar@{-}[u] \\
*{} \ar@{-}[u] \\
}}
\end{equation*}  

\noindent
and $\chi_U(p)=v$ if and only if $p \in U$. In this way, we obtain an isomorphism of partial orders $\mathcal O(P) \to Hom(P,L_2)$. 
We also recall the familiar natural order isomorphism for three partial orders $P$, $Q$ and $R$,
$$Hom (P \times Q, R) \to Hom(P, Hom(Q,R)).$$
Finally, we note that $\mathcal O (P)$ is the free complete semilattice generated by the basic open sets $B(P)$, which form a partial order isomorphic to $P^\vee$.
More precisely, if $L$ is a semilattice with all suprema and $f: P^\vee \to L$ is a map of partially ordered sets, then $f$ extends uniquely to a map $\bar f: \mathcal O(P) \to L$ 
of semilattices, i.e. preserving all suprema. Using these generalities, we observe the following:

\begin{cor}
 There is a natural order isomorphism
$$Sh(S,T) \to Hom_\vee (\mathcal O (V(T)), \mathcal O(V(S)))$$
between the shuffles and the maps of sup-semilattices.
\end{cor}

\begin{proof}
We just rewrite the isomorphism of the proposition, as follows.
\begin{eqnarray*}
 Sh(S,T)&=&\mathcal O (V(S) \times V(T)^\vee )\\
&=& Hom (V(S) \times V(T)^\vee, L_2)\\
&=& Hom ( V(T)^\vee, Hom(V(S),L_2))\\
&=& Hom ( V(T)^\vee, \mathcal O (V(S)))\\
&=& Hom_\vee ( \mathcal O(V(T)), \mathcal O (V(S))
\end{eqnarray*}
\end{proof}

\begin{cor}
For any three trees $R$, $S$ and $T$, there is an asssociative composition
$$Sh(S,T) \times Sh (R,S) \to Sh(R,T)$$
giving the trees and shuffles between them the structure of a category
(more precisely, a category enriched in distributive lattices).
\end{cor}

We conclude this section with a discussion of the group of automorphisms of the lattice $Sh(S,T)$ of the shuffles of two trees $S$ and $T$; or equivalently,
the group of automorphisms of the partial order $V(S) \times V(T)^\vee$. Clearly two automorphisms $\alpha$ of $S$ and $\beta$ of $T$ give an automorphism 
of $Sh(S,T)$, thus defining a map
$$\pi: Aut(S) \times Aut(T) \to Aut(Sh(S,T)).$$
We shall prove the following result.

\begin{prop}
Let $S$ and $T$ be two reduced trees. Then the above map  $\pi: Aut(S) \times Aut(T) \to Aut(Sh(S,T))$ is an isomorphism, 
except in the case where $S=T$ is a linear tree with at least two vertices, in which case $Aut(S)=Aut(T)$ is trivial while $Aut(Sh(S,T))$ is $\Sigma_2$.
\end{prop}

\begin{proof}
 Let us start with the easy linear case.
If $S=L_m$ and $T=L_n$ are linear trees with $m$ and $n$ vertices, respectively, than $V(S) \times V(T)^\vee$ is a rectangle $\{1, \ldots, m\} \times \{ 1, \ldots, n\}^\vee$.
The only possible non-identity automorphism of such a rectangle switches the two initial sides $\{1\} \times \{ 1, \ldots, n\}^\vee$ and $\{1, \ldots, m\} \times \{ n\}$, and this
automorphism exists if and only if $n=m$ (but is again the identity if $n=m=1$).

Next, suppose $S$ is not linear, and let $s_1$ and $s_2$ be two distinct top vertices in $S$. Also, write
$$M=\{(s,r_T) \, | \, s \text{ is a top vertex in } S \}$$
for the set of maximal elements in V(S) $\times V(T)^\vee$.
The subset $V(S) \times \{r_T\} \subset V(S) \times V(T)^\vee$ contains $M$, is open, and is a tree.
We claim it is the maximal set with these three properties. Indeed, if $A$ is an open tree containing $M$, then $A$ cannot contain any element of the form $(s',t')$ with $t' \neq r_T$.
To see this, observe that since $A$ is a tree, it must then contain a lower bound for $(s_1, r_T), (s_2, r_T)$ and $(s',t')$, which is an element $(s,t)$ with $s \leq s_1, s_2$ and $t\neq r_T$. Since $A$ is open, it will then also contain $(s_1,t)$ and $(s_2,t)$, as well as $(s,r_T)$:

$$
\xymatrix{
(s_1,r_T) \ar@{-}[dd] \ar@{-}[dr] & & (s_2,r_T) \ar@{-}[dd] \ar@{-}[dl] \\
& (s,r_T)\ar@{-}[dd] & \\
(s_1,t)  \ar@{-}[dr] & & (s_2,t)  \ar@{-}[dl] \\
& (s,t)
}
$$
which is impossible in a tree.
This shows that $A=V(S) \times \{r_T\}$, and hence this set is completely described by order-theoretic properties. In particular, any automorphism $\psi$ of $V(S) \times V(T)^\vee$
must map $V(S) \times \{r_T\}$ to itself, and hence fix $(r_S, r_T)$. Thus $\psi$ restricts to an automorphism of $\uparrow (r_S,r_T) \simeq V(S)$ and of $\downarrow (r_S,r_T) \simeq V(T)^\vee$.
In other words, there are unique automorphisms $\alpha$ of $S$ and $\beta$ of $T$ such that 
$$\psi(r_S,t)= \beta(t) \ \text{ and } \psi(s,r_T)=\alpha(s).$$ 
We claim that $\psi(s,t)=(\alpha(s), \beta(t))$ for all $s$ and $t$. Indeed, such a point $(s,t)$ lies on the side of the rectangular sublattice defined as the ``interval'' 
between $(r_S,t)$ and $(s,r_T)$.

$$
\xymatrix{
& (s,r_T)\ar@{-}[dr]\ar@{-}[dl] & \\
(r_S,r_T)  \ar@{-}[dr] & & (s,t)  \ar@{-}[dl] \\
& (r_S,t)
}
$$
But we just saw that the only possible automorphism of such a rectangle switches $(r_S,r_T)$ and $(s,t)$. Applying this to $\psi^{-1} \circ \pi (\alpha, \beta)$,
which fixes $(s,r_T)$ and $(r_S,t)$ and hence defines an automorphism of this rectangle, we find that since it fixes $(r_S,r_T)$, it must also fix $(s,t)$. 
This proves $\psi(s,t)=(\alpha(s), \beta(t))$ for all $s$ and $t$.

The case where $T$ is non-linear follows by symmetry.
\end{proof}

%

\section{Relation to other structures}\label{S:Last}

In this section, we briefly sketch the r\^ole played by shuffles in the relation between trees and other structures, viz. topological spaces, operads and dendroidal sets.

\subsection{Topological spaces}

Each tree $S$ defines a topological space $BS$, called its classifying space. One way to define it is by viewing the partial order $E(S)$ as a category, and then to take the classifying space
of this category, i.e. the geometric realization of its nerve $N(E(S))$ \cite{GJ}. Explicitly, a point of $BS$ is an assignment $\lambda$ of a length $\lambda(e)$ to each edge $e$ of $S$.
These lengths are real numbers $ 0 \leq \lambda(e) \leq 1$, and the assignment must satisfy two conditions: 
\begin{itemize}
 \item  $\displaystyle{ \sum_e \lambda(e)=1}$
 \item the set of edges $e$ with $\lambda(e)>0$ lie on a single branch.
\end{itemize}
The topology of $BS$ is inherited from the product topology on $[0,1]^{E(S)}$. Thus, for the linear tree $L_n$ with $n+1$ edges, 
$BL_n=\{(t_0, \ldots, t_n) \, | \, 0 \leq t_i \leq 1, \sum t_i=1 \}$ is the standard $n$-simplex $\Delta^n$. 
As pointed out in the introduction, an elementary but important observation is that the product $\Delta^n \times \Delta^m$ is a
union of copies of $\Delta^{n+m}$ indexed by all the $(n,m)$-shuffles, i.e. all the shuffles of $L_n$ and $L_m$. More generally, one has the following proposition.

\begin{prop}
Let $S$ and $T$ be trees. Then $BS \times BT = \bigcup_A BA$ where $A$ ranges over all the shuffles of $S$ and $T$.
\end{prop}

\begin{proof}
 Each shuffle $A$ of the trees $S$ and $T$ defines an injection of posets $E(A) \to E(S) \times E(T)$, and hence a map 
$|NE(A)| \to |N(E(S \times T))|$.
Since realization and nerve commute with products and preserve injections, this gives an embedding of compact Hausdorff spaces $BA \to BS \times BT$.
The family of all these embeddings is jointly surjective. Indeed, the family of maps of simplicial sets $NE(A) \to N(E(S) \times E(T))$ already is.
To see this, consider a non-degenerate $n$-simplex of $N(E(S) \times E(T))$, say
\begin{equation*}
 (s_0,t_0) \leq (s_1,t_1) \leq \ldots \leq (s_n,t_n).
\end{equation*}
Call such a non-degenerate simplex maximal if it is not a face of another non-degenerate simplex. Each such maximal simplex must have the property that
$(s_n,t_n)$ is a pair of leaves, while $(s_0,t_0)$ is the pair of roots, and the simplex itself is a shuffle of the two branches, down from $s_n$ in $S$ and down from $t_n$ in $T$.
Thus, a maximal simplex is a simplex in a shuffle $A$ of $S$ and $T$ (cf Remark~\ref{rem:Extension}).
\end{proof}

\begin{remark}\label{Rem:ColimitDiagram}
Realization and nerve also preserve pullbacks (intersections), so for two shuffles $A$ and $A'$, we find that $BA \cap BA' = B(A \cap A')$.
Notice that the partial suborder $E(A) \cap E(A') \subset ES \times E(T)$ is again a tree with the same root and leaves as $A$ and as $A'$. 
This tells us exactly how the different $BA$'s are glued together to form $BS \times BT$. 
More formally, let $\mathcal I$ be the family of finite non-empty subsets $I=\{A_1, \ldots, A_n \}$ of the set of all shuffles $A$ of $S$ and $T$, and let 
$A_I= A_1 \cap \ldots \cap A_n$. Then $BS \times BT$ is the colimit of the diagram of spaces $BA_I$ indexed by $\mathcal I$.
\end{remark}

\subsection{Operads}
A coloured operad $P$ consists of a set $C=C_P$ of colours, and for each $n \geq 0$ and each $(n+1)$-tuple $(c_1, \ldots, c_n,c)$ of elements of $C$
a set $P(c_1, \ldots, c_n; c)$ of ``operations''.
 The idea is that these operations take $n$ inputs of ``colours'' $c_1, \ldots, c_n$ respectively, and produce an output of colour $c$.
For example, if $\{X_c\}_{c \in C}$ is a family of sets indexed by $C$, one can take $P(c_1, \ldots, c_n; c)$  to be the set of functions
$X_{c_1} \times \ldots \times X_{c_n} \to X_c $.
By definition, an operad also has structure maps, encoding  the composition of operations, identity operations and permutations of variables.
We refer to \cite{Yau} for a precise definition and examples.

For two operads $P$ and $Q$, with sets of colours $C$ and $D$ respectively, there is a so-called \textit{Boardman-Vogt tensor product} $P \otimes_{BV} Q$, 
which is a coloured operad with $C \times D$ a set of colours, and operations generated by the following:
$$ p \otimes d \in (P \otimes_{BV} Q)((c_1,d), \ldots, (c_n,d); (c,d))$$
for each $p \in P(c_1, \ldots, c_n; c)$ and $d\in D$; and
$$c \otimes q \in (P \otimes_{BV} Q)((c,d_1), \ldots, (c,d_n) ; (c,d))$$
for each $q\in Q(d_1, \ldots, d_n;d)$ and $c \in C$.
These generators are next subject to relations coming from the structure maps of $P$ and $Q$ separately, together with the following
\textit{Boardman-Vogt (interchange) relation}:
$$(p \otimes d)(c_1 \otimes q, \ldots, c_n \otimes q ; c \otimes d)=(c\otimes q)(p\otimes d_1, \ldots, p \otimes d_m ; c \otimes d) \cdot \tau$$
for $p$ and $q$ as above, and $\tau$ the permutation putting the input colours $(c_i,d_j)$ in the same order.
See \cite{BV}  for a detailed description.

Every tree $T$ defines an operad $\Omega(T)$ whose colours are the edges of $T$. There is exactly one operation in $\Omega(T)(e_1, \ldots, e_n ; e)$ if there is
a subtree of $T$ with leaves $e_1, \ldots, e_n $ and root edge $e$ ; and otherwise $\Omega(T)(e_1, \ldots, e_n ; e)$ is empty.
Composition of operations is given by the grafting of subtrees.

For two trees $S$ and $T$, there is a diagram of trees indexed by the poset $\mathcal I$ from Remark~\ref{Rem:ColimitDiagram}, assigning to 
$\{A_1, \ldots, A_n \}$ the tree $A_I=A_1 \cap \ldots \cap A_n$. For each of these trees, there is a injective map of operads
$$\Omega(A_I) \to \Omega(S) \otimes_{BV} \Omega(T)$$
\noindent and in fact, $\Omega(S) \otimes_{BV} \Omega(T)$ is the union of the $\Omega(A)$'s indexed by all the shuffles of $S$ and $T$ ;
or more precisely, the colimit of the diagram of operads $\Omega(A_I)$ induced by the sets $I$ in $\mathcal I$.

\subsection{Dendroidal sets}
There is a category $\Omega$ introduced in \cite{MW1} whose objects are trees (possibly with stumps), and whose morphisms $S \to T$ are maps of operads $\Omega(S) \to \Omega(T)$.
The category of presheaves of sets on $\Omega$, i.e. of functors $X : \Omega^{op} \to Sets$, 
and natural transformations between them is referred to as the category of \textit{dendroidal sets}, and denoted $dSets$.
It carries a tensor product, defined on two representable dendroidal sets $\Omega[S]$ and $\Omega[T]$ corresponding to trees $S$ and $T$ by
$$(\Omega[S] \otimes \Omega[T])(R)=Hom(\Omega(R), \Omega(S) \otimes_{BV} \Omega(T))$$
\noindent where $R$ is any object of $\Omega$ and $Hom$ is in the category of coloured operads.
This tensor product is then extended to arbitrary dendroidal sets in the unique (up to isomorphism) way so as to preserve colimits in each variable separately.
Every presheaf is a colimit of representables, and for the presheaf $\Omega[S] \otimes \Omega[T]$ just defined, one can check that
$\Omega[S] \otimes \Omega[T]$ is the colimit of $\Omega[A]$'s where $A$ ranges over finite intersections of shuffles, exactly as in Remark~\ref{Rem:ColimitDiagram} above.
This fact plays a crucial r\^ole in the analysis of homotopical properties of the tensor product of dendroidal sets, see e.g. \cite{HHM}.
It follows formally that for the unique colimit preserving functor $B$ from dendroidal sets to topological spaces defined on representables by
$B(\Omega[T])=BE(T)$ (the classifying space of $E(T)$), there is a natural homeomorphism 
$$B(X \otimes Y) \to B(X) \times B(Y)$$
 for any two dendroidal sets $X$ and $Y$.
This property, in turn, is closely related to the fact that $\Omega$ is a ``test category'' in the sense of Grothendieck, cf \cite{ACM}.

\bibliographystyle{plain}

\bibliography{ShufflesOfTrees}

\end{document}